\documentclass{amsart}
\usepackage{amsfonts}
\usepackage{hyperref}

\newtheorem{theorem}{Theorem}
\theoremstyle{plain}

\newtheorem{corollary}{Corollary}

\newtheorem{lemma}{Lemma}

\numberwithin{equation}{section}
\input{tcilatex}

\begin{document}
\title[Sharp bounds for Neuman-S\'{a}ndor mean]{Sharp power means bounds for
Neuman-S\'{a}ndor mean}
\author{Zhen-Hang Yang}
\address{System Division, Zhejiang Province Electric Power Test and Research
Institute, Hangzhou, Zhejiang, China, 31001}
\email{yzhkm@163.com}
\date{July 5, 2012}
\thanks{This paper is in final form and no version of it will be submitted
for publication elsewhere.}

\begin{abstract}
For $a,b>0$ with $a\neq b$, let $N\left( a,b\right) $ denote the Neuman-S%
\'{a}ndor mean defined by 
\begin{equation*}
N\left( a,b\right) =\frac{a-b}{2\func{arcsinh}\frac{a-b}{a+b}}
\end{equation*}%
and $A_{r}\left( a,b\right) $ denote the $r$-order power mean. We present
the sharp power means bounds for the Neuman-S\'{a}ndor mean: 
\begin{equation*}
A_{p_{1}}\left( a,b\right) <N\left( a,b\right) \leq A_{p_{2}}\left(
a,b\right) ,
\end{equation*}%
where $p_{1}=$ $\frac{\ln 2}{\ln \ln \left( 3+2\sqrt{2}\right) }$ and $%
p_{2}=4/3$ are the best constants.
\end{abstract}

\maketitle

\section{Introduction}

Throughout the paper, we assume that $a,b>0$ with $a\neq b$. The classical
power mean of order $r$ of the positive real numbers $a$ and $b$ is defined
by

\begin{equation*}
A_{r}=A_{r}(a,b)=\left( \frac{a^{r}+b^{r}}{2}\right) ^{1/r}\text{ if }r\neq 0%
\text{ and }A_{0}=A_{0}(a,b)=\sqrt{ab}.
\end{equation*}%
It is well-known that the function $r\mapsto A_{r}(a,b)$ is continuous and
strictly increasing on $\mathbb{R}$ (see \cite{Bullen.1988}). As special
cases, the arithmetic mean, geometric mean and quadratic mean are $A=A\left(
a,b\right) =A_{1}\left( a,b\right) $, $G=G\left( a,b\right) =A_{0}\left(
a,b\right) $ and $Q=Q\left( a,b\right) =A_{2}\left( a,b\right) $,
respectively.

The logarithmic mean and identrice (exponential) mean are defined as%
\begin{eqnarray*}
L &=&L\left( a,b\right) =\frac{a-b}{\ln a-\ln b}, \\
I &=&I\left( a,b\right) =e^{-1}\left( a^{a}/b^{b}\right) ^{1/\left(
a-b\right) },
\end{eqnarray*}%
respectively. In 1993, Seiffert \cite{Seiffert.4(11)(1993)} introduced his
first mean as 
\begin{equation}
P=P\left( a,b\right) =\frac{a-b}{4\arctan \sqrt{a/b}-\pi },  \label{P1}
\end{equation}%
which can be written also in the equivalent form 
\begin{equation}
P=P\left( a,b\right) =\frac{a-b}{2\arcsin \frac{a-b}{a+b}},  \label{P2}
\end{equation}%
see e.g. \cite{Seiffert.42(1987)}. In 1995, Seiffert \cite{Seiffert.29(1995)}
defined his second mean as%
\begin{equation*}
T=T\left( a,b\right) =\frac{a-b}{2\arctan \frac{a-b}{a+b}}.
\end{equation*}%
Recently, Neuman and S\'{a}ndor have defined in \cite{Neuman.17(1)(2006)} a
new mean%
\begin{equation}
N=N\left( a,b\right) =\frac{a-b}{2\func{arcsinh}\frac{a-b}{a+b}}=\frac{a-b}{%
2\ln \frac{a-b+\sqrt{2\left( a^{2}+b^{2}\right) }}{a+b}}  \label{N-S mean}
\end{equation}

All these means are symmetric and homogeneous, and the power mean is
relatively simple. Hence ones are interested in evaluating these means by
power means $A_{p}$.

Ostle and Terwilliger \cite{Ostle.17.1957} and Karamata \cite%
{Karamata.24.1960} first proved that 
\begin{equation}
G<L<A.  \label{O-T-K}
\end{equation}%
This result, or a part of it, has been rediscovered and reproved many times
(see e. g., \cite{Mitrinovic.1970}, \cite{Yang. 1984}, \cite{Yang.4.1987}, 
\cite{Sandor.43.1988}). In 1974 Lin \cite{Lin.79.1972} obtained an important
refinement of the above inequalities:%
\begin{equation}
G<L<A_{1/3},  \label{Lin}
\end{equation}%
and proved that the number $1/3$ cannot be replaced by a smaller one.

For the identric mean $I$, Stolarsky \cite{Stolarsky.48.1975} first proved
that 
\begin{equation}
G<I<A  \label{Stolarsky 1}
\end{equation}%
(also see \cite{Yang. 1984}, \cite{Yang.4.1987}). In 1988, Alzer \cite%
{Alzer.43.1988} showed that 
\begin{equation}
2e^{-1}A<I<A  \label{Alzer}
\end{equation}%
(also see \cite{Sandor.189(1995)}). The following double inequality 
\begin{equation}
A_{1/2}<I<4e^{-1}A_{1/2}  \label{N-E 1}
\end{equation}%
is due to Neuman and S\'{a}ndor \cite{Neuman.16(2003)}. Stolarsky \cite%
{Stolarsky.87.1980} and Pittenger \cite{Pittinger.680.1980} established the
sharp lower and upper bounds for $I$ in terms of power means 
\begin{equation}
A_{2/3}<I<A_{\ln 2},  \label{S-P}
\end{equation}%
respectively. By using the well properties of homogeneous functions, Yang
also proved (\ref{Alzer}), (\ref{N-E 1}) in \cite{Yang.6(4).101.2005} and 
\begin{equation}
A_{2/3}<I<2\sqrt{2}e^{-1}A_{2/3}  \label{S-Y}
\end{equation}%
in \cite{Yang.10(3).2007}.

For the first Seiffert mean $P$, the author \cite{Seiffert.4(11)(1993)} gave
a estimate by $A$%
\begin{equation}
\frac{2}{\pi }A<P<A.  \label{Seiffert 1}
\end{equation}%
Subsequently, Jagers \cite{Jagers.12(1994)} proved that%
\begin{equation}
A_{1/2}<P<A_{2/3}.  \label{Jagers}
\end{equation}%
By using Pfaff's algorithm S\'{a}dor in \cite{Sandor.76(2001)} reproved the
first inequality in (\ref{Jagers}), while H\"{a}sto \cite{Hasto.3(5)(2002)}
gave a companion one of the second one in (\ref{Jagers}): 
\begin{equation}
\frac{2\sqrt{2}}{\pi }A_{2/3}<P<A_{2/3},  \label{Hasto 1}
\end{equation}%
Two year later, H\"{a}sto \cite{Hasto.7(1)(2004)} obtained further a sharp
lower bound for $P$: 
\begin{equation}
P>A_{\ln _{\pi }2}.  \label{Hasto 2}
\end{equation}

In 1995, Seiffert \cite{Seiffert.29(1995)} showed that 
\begin{equation}
A<T<A_{2}.  \label{Seiffert 2}
\end{equation}%
Very recently, Yang \cite{Yang.arxiv.1206.5494V1} present the sharp bounds
for the second Seiffert mean in terms of power means: 
\begin{equation}
A_{\log _{\pi /2}2}<T\leq A_{5/3}.  \label{Yang 1}
\end{equation}%
Moreover, he obtained that 
\begin{eqnarray}
\alpha A_{5/3} &<&T<A_{5/3},  \label{Yang 2} \\
A_{\log _{\pi /2}2} &<&T<\beta A_{\log _{\pi /2}2},  \label{Yang 3}
\end{eqnarray}%
where $\alpha =2^{8/5}\pi ^{-1}=\allowbreak 0.964\,94...$ and $\beta
=1.5349...$ are the best possible constants.

Concerning the Neuman-S\'{a}ndor mean, the author \cite{Neuman.17(1)(2006)}
first established%
\begin{equation}
G<L<P<A<N<T<A_{2}  \label{E-S 1}
\end{equation}%
and 
\begin{equation}
\frac{\pi }{2}P>A>\func{arcsinh}\left( 1\right) N>\frac{\pi }{2}T.
\label{E-S 2}
\end{equation}%
Lately, Constin and Toader \cite[Theorem 1]{Costin.IJMMS.2012.inprint} have
shown that $A_{3/2}$ can be put between $N$ and $T$, that is, 
\begin{equation}
N<A_{3/2}<T,  \label{C-S 1}
\end{equation}%
and they obtained the following nice chain of inequalities for certain means:%
\begin{equation}
G<L<A_{1/2}<P<A<N<A_{3/2}<T<A_{2}.  \label{C-S 2}
\end{equation}

Our aim is to prove that%
\begin{equation}
A_{\frac{\ln 2}{\ln \ln \left( 3+2\sqrt{2}\right) }}<N<A_{4/3},  \label{Main}
\end{equation}%
where $\frac{\ln 2}{\ln \ln \left( 3+2\sqrt{2}\right) }$ and $4/3$ are the
best possible constants. Thus, we obtain a more nice chain of inequalities
for bivariate means: 
\begin{eqnarray*}
A_{0} &<&L<A_{1/3}<A_{\ln _{\pi }2}<P<A_{2/3}<I<A_{\ln 2} \\
&<&A_{\frac{\ln 2}{\ln \ln \left( 3+2\sqrt{2}\right) }}<N<A_{4/3}<A_{\log
_{\pi /2}2}<T<A_{5/3}
\end{eqnarray*}

Our main results are the following

\begin{theorem}
\label{Theorem 1}For $a,b>0$ with $a\neq b$, the inequality $N<A_{p}$ holds
if and only if $p\geq 4/3$. Moreover, we have 
\begin{equation}
\allowbreak \alpha _{1}A_{4/3}<N<\beta _{1}A_{4/3},  \label{Ma}
\end{equation}%
where $\alpha _{1}=\tfrac{1}{\sqrt[4]{2}\ln \left( \sqrt{2}+1\right) }%
=\allowbreak 0.954\,07...$ and $\beta _{1}=1$ are the best possible
constants.
\end{theorem}

\begin{theorem}
\label{Theorem 2}For $a,b>0$ with $a\neq b$, the inequality $N>A_{p}$ holds
if and only if $p\leq p_{0}=\frac{\ln 2}{\ln \ln \left( 3+2\sqrt{2}\right) }%
\approx \allowbreak 1.\,\allowbreak 222\,8$. Moreover, we have 
\begin{equation}
\allowbreak \alpha _{2}A_{p_{0}}<N<\beta _{2}\allowbreak A_{p_{0}},
\label{Mb}
\end{equation}%
where $\allowbreak \alpha _{2}=1$ and $\beta _{2}\approx 1.\,\allowbreak
013\,8$ are the best possible constants.
\end{theorem}

\section{Lemmas}

In order to prove our main results, we need the following lemmas.

\begin{lemma}
\label{Lemma 2.1}Let $F_{p}$ be the function defined on $\left( 0,1\right) $
by 
\begin{equation}
F_{p}\left( x\right) =\ln \frac{N\left( 1,x\right) }{A_{p}\left( 1,x\right) }%
=\ln \frac{x-1}{2\ln \frac{x-1+\sqrt{2\left( x^{2}+1\right) }}{x+1}}-\frac{1%
}{p}\ln \left( \frac{x^{p}+1}{2}\right) .  \label{F_p}
\end{equation}%
Then we have 
\begin{eqnarray}
\lim_{x\rightarrow 1^{-}}\frac{F_{p}\left( x\right) }{\left( x-1\right) ^{2}}
&=&\allowbreak -\frac{1}{24}\left( 3p-4\right) ,  \label{2.1} \\
F_{p}\left( 0^{+}\right) &=&\lim_{x\rightarrow 0^{+}}F_{p}\left( x\right)
=\left\{ 
\begin{array}{lc}
\frac{1}{p}\ln 2-\ln \ln \left( 3+2\sqrt{2}\right) & \text{if }p>0, \\ 
\infty & \text{if }p\leq 0,%
\end{array}%
\right.  \label{2.2}
\end{eqnarray}%
where $F_{0}\left( x\right) :=\lim_{p\rightarrow 0}F_{p}\left( x\right) $.
\end{lemma}

\begin{proof}
Using power series expansion we have 
\begin{equation*}
F_{p}\left( x\right) =\allowbreak -\frac{1}{24}\left( 3p-4\right) \left(
x-1\right) ^{2}+O\left( \left( x-1\right) ^{3}\right) ,
\end{equation*}%
which yields (\ref{2.1}).

Direct limit calculation leads to (\ref{2.2}), which proves the lemma.
\end{proof}

\begin{lemma}
\label{Lemma 2.2}Let $F_{p}$ be the function defined on $\left( 0,1\right) $
by (\ref{F_p}). Then $F_{p}$ is strictly increasing on $\left( 0,1\right) $
if and only if $p\geq 4/3$ and decreasing on $\left( 0,1\right) $\ if and
only if $p\leq 1$.
\end{lemma}

\begin{proof}
Differentiation yields%
\begin{equation}
F_{p}^{\prime }\left( x\right) =\frac{x^{p-1}+1}{x^{p}+1}\frac{1}{\left(
x-1\right) \ln \frac{x-1+\sqrt{2\left( x^{2}+1\right) }}{x+1}}\times
f_{p}\left( x\right) ,  \label{dF_p}
\end{equation}%
where 
\begin{equation}
f_{p}\left( x\right) =\ln \frac{x-1+\sqrt{2\left( x^{2}+1\right) }}{x+1}-%
\sqrt{2}\frac{x-1}{\left( x+1\right) \sqrt{x^{2}+1}}\frac{x^{p}+1}{x^{p-1}+1}
\label{f_p}
\end{equation}%
Differentiating $f_{p}\left( x\right) $ and simplifying lead to%
\begin{equation}
f_{p}^{\prime }\left( x\right) =\allowbreak \frac{\sqrt{2}\left( 1-x\right)
x^{p}}{\left( \sqrt{x^{2}+1}\right) ^{3}\left( x+1\right) ^{2}\left(
x+x^{p}\right) ^{2}}g\left( x\right) ,  \label{df_p}
\end{equation}%
where 
\begin{equation}
g\left( x\right) =\left(
x^{p+2}+x^{p+1}+2x^{p}-x^{2-p}-x^{3-p}-2x^{4-p}+\left( p-1\right)
x^{4}-x^{3}+x-p+1\right) .  \label{g}
\end{equation}%
(i) We now prove that $F_{p}$ is strictly increasing on $\left( 0,1\right) $
if and only if $p\geq 4/3$. From (\ref{dF_p}) it is seen that $\func{sgn}%
F_{p}^{\prime }\left( x\right) =\func{sgn}$ $f_{p}\left( x\right) $ for $%
x\in \left( 0,1\right) $, so it suffices to prove that $f_{p}\left( x\right)
>0$ for $x\in \left( 0,1\right) $ if and only if $p\geq 4/3$.

\textbf{Necessity}. If $f_{p}\left( x\right) >0$ for $x\in \left( 0,1\right) 
$ then there must be $\lim_{x\rightarrow 1^{-}}\left( 1-x\right)
^{-3}f_{p}\left( x\right) \geq 0$. Application of L'Hospital rule leads to 
\begin{equation*}
\lim_{x\rightarrow 1^{-}}\frac{f_{p}\left( x\right) }{\left( 1-x\right) ^{3}}%
=\lim_{x\rightarrow 1^{-}}\frac{\ln \frac{x-1+\sqrt{2\left( x^{2}+1\right) }%
}{x+1}-\sqrt{2}\frac{x-1}{\left( x+1\right) \sqrt{x^{2}+1}}\frac{x^{p}+1}{%
x^{p-1}+1}}{\left( 1-x\right) ^{3}}=\allowbreak \frac{1}{8}\left( p-\frac{4}{%
3}\right) ,
\end{equation*}%
and so we have $p\geq 4/3$.

\textbf{Sufficiency}. We now prove $f_{p}\left( x\right) >0$ for $x\in
\left( 0,1\right) $ if $p\geq 4/3$. Since the\ Lehmer mean of order $r$ of
the positive real numbers $a$ and $b$ defined as 
\begin{equation}
\mathcal{L}_{r}=\mathcal{L}_{r}\left( a,b\right) =\frac{a^{r+1}+b^{r+1}}{%
a^{r}+b^{r}}  \label{Lehmer}
\end{equation}%
(see \cite{Lehmer.36(1971)}) is increasing in its parameter on $\mathbb{R}$,
it is enough to show that $f_{p}\left( x\right) >0$ for $x\in \left(
0,1\right) $ when $p=4/3$. In this case, we have 
\begin{equation*}
g\left( x\right) =x-x^{3}+\frac{1}{3}x^{4}-x^{\frac{2}{3}}+2x^{\frac{4}{3}%
}-x^{\frac{5}{3}}+x^{\frac{7}{3}}-2x^{\frac{8}{3}}+x^{\frac{10}{3}%
}-\allowbreak \frac{1}{3},
\end{equation*}%
and therefore 
\begin{equation*}
3g\left( x^{3}\right)
=x^{12}+3x^{10}-3x^{9}-6x^{8}+3x^{7}-3x^{5}+6x^{4}+3\allowbreak
x^{3}-3x^{2}-1.
\end{equation*}%
Factoring yields that for $x\in \left( 0,1\right) $ 
\begin{equation*}
3g\left( x^{3}\right) =\left( x-1\right) ^{3}\left( x+1\right) \left(
x^{8}+2x^{7}+7x^{6}+9x^{5}+9x^{4}+9x^{3}+7x^{2}+2x+\allowbreak 1\right) <0.
\end{equation*}

It follows from (\ref{df_p}) that $f_{p}^{\prime }\left( x\right) <0$, that
is, the function $f_{p}$ is decreasing on $\left( 0,1\right) $. Hence for $%
x\in \left( 0,1\right) $ we have $f_{p}\left( x\right) >f_{p}\left( 1\right)
=0$, which proves the sufficiency.

(ii) We next prove that $F_{p}$ is strictly decreasing on $\left( 0,1\right) 
$ if and only if $p\leq 1$. Similarly, it suffices to show that $f_{p}\left(
x\right) <0$ for $x\in \left( 0,1\right) $ if and only if $p\leq 1$.

\textbf{Necessity}. If $f_{p}\left( x\right) <0$ for $x\in \left( 0,1\right) 
$ then we have 
\begin{equation*}
\lim_{x\rightarrow 0^{+}}f_{p}\left( x\right) =\left\{ 
\begin{array}{ll}
\ln \left( \sqrt{2}-1\right) +\sqrt{2} & \text{if }p>1 \\ 
\ln \left( \sqrt{2}-1\right) +\frac{\sqrt{2}}{2} & \text{if }p=1 \\ 
\ln \left( \sqrt{2}-1\right) & \text{if }p<1%
\end{array}%
\right. \leq 0,
\end{equation*}%
which yields $p\leq 1$.

\textbf{Sufficiency}. We prove $f_{p}\left( x\right) <0$ for $x\in \left(
0,1\right) $ if $p\leq 1$. As mentioned previous, the function $p\mapsto
L_{p-1}\left( 1,x\right) $ is increasing on $\mathbb{R}$, it suffices to
demonstrate $f_{p}\left( x\right) <0$ for $x\in \left( 0,1\right) $ when $%
p=1 $. In this case, we have $g\left( x\right) =\allowbreak \allowbreak
2x-2x^{3}>0$, then $f_{p}^{\prime }\left( x\right) >0$, and then for $x\in
\left( 0,1\right) $ we have $f_{p}\left( x\right) <f_{p}\left( 1\right) =0$,
which proves the sufficiency and the proof of this lemma is finished.
\end{proof}

\begin{lemma}
\label{Lemma 2.3}Let the function $g$ be defined on $\left( 0,1\right) $ by (%
\ref{g}). Then there is a unique a $x_{0}\in \left( 0,1\right) $ such that $%
g\left( x\right) <0$ for $x\in \left( 0,x_{0}\right) $ and $g\left( x\right)
>0$ for $x\in \left( x_{0},1\right) $ if $p=p_{0}=\frac{\ln 2}{\ln \ln
\left( 3+2\sqrt{2}\right) }\in \left( 122/100,4/3\right) $.
\end{lemma}

\begin{proof}
We prove desired result stepwise.

\textbf{Step 1}: We have $g^{\left( 4\right) }\left( x\right) >0$ for $x\in
\left( 0,1\right) $ when $p\in \left( 1,4/3\right) $.

Differentiations yield 
\begin{eqnarray}
g^{\prime }\left( x\right) &=&\left( p+2\right) x^{p+1}+\left( p+1\right)
x^{p}+2px^{p-1}+\left( p-2\right) x^{1-p}  \label{dg} \\
&&+\left( p-3\right) x^{2-p}+2\left( p-4\right) x^{3-p}+4\left( p-1\right)
x^{3}-3x^{2}+1,  \notag
\end{eqnarray}

\begin{eqnarray}
g^{\prime \prime }\left( x\right) &=&\left( p+1\right) \left( p+2\right)
x^{p}+p\left( p+1\right) x^{p-1}+2p\left( p-1\right) x^{p-2}  \label{ddg} \\
&&-\left( p-1\right) \left( p-2\right) x^{-p}-\left( p-2\right) \left(
p-3\right) x^{1-p}  \notag \\
&&-2\left( p-3\right) \left( p-4\right) x^{2-p}+12\left( p-1\right) x^{2}-6x,
\notag
\end{eqnarray}

\begin{eqnarray}
g^{\prime \prime \prime }\left( x\right) &=&p\left( p+1\right) \left(
p+2\right) x^{p-1}+p\left( p-1\right) \left( p+1\right) x^{p-2}  \label{dddg}
\\
&&+2p\left( p-1\right) \left( p-2\right) x^{p-3}+p\left( p-1\right) \left(
p-2\right) x^{-p-1}  \notag \\
&&+\left( p-1\right) \left( p-2\right) \left( p-3\right) x^{-p}+2\left(
p-2\right) \left( p-3\right) \left( p-4\right) x^{1-p}  \notag \\
&&+24\left( p-1\right) x-6,  \notag
\end{eqnarray}

\begin{eqnarray}
\frac{g^{\left( 4\right) }\left( x\right) }{p-1} &=&p\left( p+1\right)
\left( p+2\right) x^{p-2}  \label{ddddg} \\
&&+p\left( p+1\right) \left( p-2\right) x^{p-3}+2p\left( p-2\right) \left(
p-3\right) x^{p-4}  \notag \\
&&-p\left( p+1\right) \left( p-2\right) x^{-p-2}-p\left( p-2\right) \left(
p-3\right) x^{-p-1}  \notag \\
&&-2\left( p-2\right) \left( p-3\right) \left( p-4\right) x^{-p}+24  \notag
\\
&:&=I_{1}+I_{2}+I_{3}+I_{4},  \notag
\end{eqnarray}%
where 
\begin{eqnarray*}
I_{1} &=&p\left( p+1\right) \left( p+2\right) x^{p-2}>0, \\
I_{2} &=&p\left( p+1\right) \left( p-2\right) x^{p-3}+2p\left( p-2\right)
\left( p-3\right) x^{p-4} \\
&=&p\left( 2-p\right) x^{p-4}\left( 2\left( 3-p\right) -\left( p+1\right)
x\right) \\
&>&p\left( 2-p\right) x^{p-3}\left( 2\left( 3-p\right) -\left( p+1\right)
\right) =p\left( 2-p\right) x^{p-3}\left( 5-3p\right) >0, \\
I_{3} &=&-p\left( p+1\right) \left( p-2\right) x^{-p-2}-p\left( p-2\right)
\left( p-3\right) x^{-p-1} \\
&=&p\left( 2-p\right) x^{-p-2}\left( \left( p+1\right) -\left( 3-p\right)
x\right) \\
&>&p\left( 2-p\right) x^{-p-2}\left( \left( p+1\right) -\left( 3-p\right)
\right) =2p\left( 2-p\right) x^{-p-2}\left( p-1\right) >0, \\
I_{4} &=&2\left( 2-p\right) \left( 3-p\right) \left( 4-p\right) x^{-p}+24>0
\end{eqnarray*}%
Hence, $g^{\left( 4\right) }\left( x\right) >0$ for $x\in \left( 0,1\right) $
when $p\in \left( 1,4/3\right) $.

\textbf{Step 2: }There is unique $x_{3}\in \left( 0,1\right) $ such that $%
g^{\prime \prime \prime }\left( x\right) <0$ for $x\in \left( 0,x_{3}\right) 
$ and $g^{\prime \prime \prime }\left( x\right) >0$ for $x\in \left(
x_{3},1\right) $ when $p\in \left( 122/100,4/3\right) $.

Since $g^{\left( 4\right) }\left( x\right) >0$ for $x\in \left( 0,1\right) $
when $p\in \left( 1,4/3\right) $, to prove this step, it suffices to verify
that $g^{\prime \prime \prime }\left( 0^{+}\right) <0$ and $g^{\prime \prime
\prime }\left( 1\right) >0$. Simple computation yields 
\begin{eqnarray*}
\func{sgn}g^{\prime \prime \prime }\left( 0^{+}\right) &=&\func{sgn}\left(
p\left( p-1\right) \left( p-2\right) \right) <0, \\
g^{\prime \prime \prime }\left( 1\right) &=&\allowbreak
8p^{3}-30p^{2}+94p-84:=h\left( p\right) >0,
\end{eqnarray*}%
where the last inequality holds is due to 
\begin{equation*}
h^{\prime }\left( p\right) =\allowbreak 24p^{2}-60p+94=\frac{3}{2}\left(
4p-5\right) ^{2}+\allowbreak \frac{113}{2}>0
\end{equation*}%
with 
\begin{equation*}
h\left( \frac{122}{100}\right) =\allowbreak \frac{17\,337}{31\,250}>0\text{
\ and }\allowbreak h\left( \frac{4}{3}\right) =\frac{188}{27}>0.
\end{equation*}

\textbf{Step 3}: There is a unique $x_{2}\in \left( 0,x_{3}\right) $ such
that $g^{\prime \prime }\left( x\right) >0$ for $x\in \left( 0,x_{2}\right) $
and $g^{\prime \prime }\left( x\right) <0$ for $x\in \left( x_{2},1\right) $
when $p\in \left( 122/100,4/3\right) $.

By Step 2 with 
\begin{eqnarray*}
\func{sgn}g^{\prime \prime }\left( 0^{+}\right) &=&\func{sgn}\left( -\left(
p-1\right) \left( p-2\right) \right) >0, \\
g^{\prime \prime }\left( 1\right) &=&\allowbreak 12\left( 3p-4\right) <0,
\end{eqnarray*}%
we see that $g^{\prime \prime }\left( x\right) <g^{\prime \prime }\left(
1\right) <0$ for $x\in \left( x_{3},1\right) $ but $g^{\prime \prime }\left(
0^{+}\right) >0$, which completes this step.

\textbf{Step 4}: There are two $x_{11}\in \left( 0,x_{2}\right) ,x_{12}\in
\left( x_{2},1\right) $ such that $g^{\prime }\left( x\right) <0$ for $x\in
\left( 0,x_{11}\right) \cup \left( x_{12},1\right) $ and $g^{\prime }\left(
x\right) >0$ for $x\in \left( x_{11},x_{12}\right) $ when $p=p_{0}=\frac{\ln
2}{\ln \ln \left( 3+2\sqrt{2}\right) }\in \left( \frac{122}{100},\frac{4}{3}%
\right) $.

Since $p=p_{0}=\frac{\ln 2}{\ln \ln \left( 3+2\sqrt{2}\right) }\in \left( 
\frac{122}{100},\frac{4}{3}\right) $, according to Step 3 and note that 
\begin{eqnarray*}
g^{\prime }\left( 0^{+}\right) &=&\func{sgn}\left( p-2\right) <0, \\
g^{\prime }\left( 1\right) &=&\allowbreak 4\left( 3p-4\right) <0,
\end{eqnarray*}%
in order to prove this step, it is enough to verify that $g^{\prime }\left(
x_{2}\right) >0$.

In fact, if $g^{\prime }\left( x_{2}\right) <0$ then $g^{\prime }\left(
x\right) <g^{\prime }\left( x_{2}\right) <0$ for $x\in \left( 0,x_{2}\right) 
$ and $g^{\prime }\left( x\right) <g^{\prime }\left( x_{2}\right) <0$ for $%
x\in \left( x_{2},1\right) $, and then $g^{\prime }\left( x\right) <0$ for $%
x\in \left( 0,1\right) $. It follows that $g\left( x\right) >g\left(
1\right) =0$, which in combination (\ref{df_p}) yields $f_{p}^{\prime
}\left( x\right) >0$. Therefore, $f_{p}\left( x\right) <f_{p}\left( 1\right)
=0$, which implies from (\ref{dF_p}) that $F_{p}^{\prime }\left( x\right) <0$%
. Then for $x\in \left( 0,1\right) $ 
\begin{equation*}
0=F_{p}\left( 0^{+}\right) >F_{p}\left( x\right) >F_{p}\left( 1\right) =0
\end{equation*}%
if $p=p_{0}=\frac{\ln 2}{\ln \ln \left( 3+2\sqrt{2}\right) }\in \left( \frac{%
122}{100},\frac{4}{3}\right) $, which is clearly a contradiction. Hence
there must be $g^{\prime }\left( x_{2}\right) >0$, which completes the Step
4.

\textbf{Step 5}: There is a unique a $x_{0}\in \left( x_{11},x_{12}\right) $
such that $g\left( x\right) <0$ for $x\in \left( 0,x_{0}\right) $ and $%
g\left( x\right) >0$ for $x\in \left( x_{0},1\right) $ if $p=p_{0}=\frac{\ln
2}{\ln \ln \left( 3+2\sqrt{2}\right) }\in \left( 122/100,4/3\right) $.

From Step 5 and notice that 
\begin{equation*}
g\left( 0^{+}\right) =1-p<0\text{, \ \ \ }g\left( 1^{-}\right) =0,
\end{equation*}%
we have the following variance table of $g\left( x\right) $:%
\begin{equation*}
\begin{tabular}{|c|c|c|c|c|c|c|c|}
\hline
$x$ & $0^{+}$ & $\left( 0,x_{11}\right) $ & $x_{11}$ & $\left(
x_{11},x_{12}\right) $ & $x_{12}$ & $\left( x_{12},1\right) $ & $1$ \\ \hline
$g^{\prime }\left( x\right) $ & $-$ & $-$ & $0$ & $+$ & $0$ & $-$ & $-$ \\ 
\hline
$g\left( x\right) $ & $-$ & $\searrow $ & $-$ & $\nearrow $ & $+$ & $%
\searrow $ & $0$ \\ \hline
\end{tabular}%
\end{equation*}%
where 
\begin{equation*}
g\left( x_{11}\right) <g\left( 0^{+}\right) =1-p<0\text{ \ \ and \ }g\left(
x_{12}\right) >g\left( 1\right) =0.
\end{equation*}%
Thus the step follows.
\end{proof}

\begin{lemma}
\label{Lemma 2.4}Let the function $f_{p}$ be defined on $\left( 0,1\right) $
by (\ref{f_p}), where $p=p_{0}=\frac{\ln 2}{\ln \ln \left( 3+2\sqrt{2}%
\right) }$. Then there is a unique $\tilde{x}_{0}\in \left( 0,x_{0}\right) $
to satisfy $f_{p}\left( \tilde{x}_{0}\right) =0$ such that $f_{p}\left(
x\right) >0$ for $x\in \left( 0,\tilde{x}_{0}\right) $ and $f_{p}\left(
x\right) <0$ for $x\in \left( \tilde{x}_{0},1\right) $.
\end{lemma}

\begin{proof}
Due to (\ref{df_p}), it is deduced that $f_{p}$ is decreasing on $\left(
0,x_{0}\right) $\ and increasing on $\left( x_{0},1\right) $, then $%
f_{p}\left( x\right) <f_{p}\left( 1\right) =0$ for $x\in \left(
x_{0},1\right) $ but $f_{p}\left( 0^{+}\right) =\ln \left( \sqrt{2}-1\right)
+\sqrt{2}>0$. This indicates that there is a unique $\tilde{x}_{0}\in \left(
0,x_{0}\right) $ to satisfy $f_{p}\left( \tilde{x}_{0}\right) =0$ such that $%
f_{p}\left( x\right) >0$ for $x\in \left( 0,\tilde{x}_{0}\right) $ and $%
f_{p}\left( x\right) <0$ for $x\in \left( \tilde{x}_{0},1\right) $.
\end{proof}

\section{Proofs of Main Results}

Based on the lemmas in the above section, we can easily proved our main
results.

\begin{proof}[Proof of Theorem \protect\ref{Theorem 1}]
By symmetry, we assume that $a>b>0$. Then inequality $N<A_{p}$ is equivalent
to%
\begin{equation}
\ln N\left( 1,x\right) -\ln A_{p}\left( 1,x\right) =F_{p}\left( x\right) <0,
\label{3.1}
\end{equation}%
where $x=b/a\in \left( 0,1\right) $. Now we prove the inequality (\ref{3.1})
holds for all $x\in \left( 0,1\right) $ if and only if $p\geq 4/3$.

\textbf{Necessity}. If inequality (\ref{3.1}) holds, then by Lemma \ref%
{Lemma 2.1} we have%
\begin{equation*}
\left\{ 
\begin{array}{l}
\lim_{x\rightarrow 1^{-}}\frac{F_{p}\left( x\right) }{\left( x-1\right) ^{2}}%
=-\frac{1}{24}\left( 3p-4\right) \leq 0, \\ 
\lim_{x\rightarrow 0^{+}}F_{p}\left( x\right) =\frac{1}{p}\ln 2-\ln \ln
\left( 3+2\sqrt{2}\right) \leq 0\text{ if }p>0,%
\end{array}%
\right.
\end{equation*}%
which yields $p\geq 4/3$.

\textbf{Sufficiency. }Suppose that $p\geq 4/3$. It follows from Lemma \ref%
{Lemma 2.2} that $F_{p}\left( x\right) <F_{p}\left( 1\right) =0$ for $x\in
\left( 0,1\right) $, which proves the sufficiency.

Using the monotonicity of the function $x\mapsto F_{4/3}\left( x\right) $ on 
$\left( 0,1\right) $, we have 
\begin{equation*}
\ln \tfrac{1}{\sqrt[4]{2}\ln \left( \sqrt{2}+1\right) }=F_{4/3}\left(
0^{+}\right) <F_{54/3}\left( x\right) <F_{4/3}\left( 1^{-}\right) =0,
\end{equation*}%
which implies (\ref{Mb}).

Thus the proof of Theorem \ref{Theorem 1} is finished.
\end{proof}

\begin{proof}[Proof of Theorem \protect\ref{Theorem 2}]
Clearly, the inequality $N>A_{p}$ is equivalent to%
\begin{equation}
\ln N\left( 1,x\right) -\ln A_{p}\left( 1,x\right) =F_{p}\left( x\right) >0,
\label{3.2}
\end{equation}%
where $x=b/a\in \left( 0,1\right) $. Now we show that the inequality (\ref%
{3.2}) holds for all $x\in \left( 0,1\right) $ if and only if $p\leq \frac{%
\ln 2}{\ln \ln \left( 3+2\sqrt{2}\right) }$.

\textbf{Necessity}. The condition $p\leq \frac{\ln 2}{\ln \ln \left( 3+2%
\sqrt{2}\right) }$ is necessary. Indeed, if inequality (\ref{3.2}) holds,
then we have%
\begin{equation*}
\left\{ 
\begin{array}{l}
\lim_{x\rightarrow 1^{-}}\frac{F_{p}\left( x\right) }{\left( x-1\right) ^{2}}%
=-\frac{1}{24}\left( 3p-4\right) \geq 0, \\ 
\lim_{x\rightarrow 0^{+}}F_{p}\left( x\right) =\frac{1}{p}\ln 2-\ln \ln
\left( 3+2\sqrt{2}\right) \geq 0\text{ if }p>0%
\end{array}%
\right.
\end{equation*}%
or%
\begin{equation*}
\left\{ 
\begin{array}{l}
\lim_{x\rightarrow 1^{-}}\frac{F_{p}\left( x\right) }{\left( x-1\right) ^{2}}%
=-\frac{1}{24}\left( 3p-4\right) \geq 0, \\ 
\lim_{x\rightarrow 0^{+}}F_{p}\left( x\right) =\infty \text{ if }p\leq 0.%
\end{array}%
\right.
\end{equation*}%
Solving the above inequalities leads to $p\leq \frac{\ln 2}{\ln \ln \left(
3+2\sqrt{2}\right) }$.

\textbf{Sufficiency. }The condition $p\leq \frac{\ln 2}{\ln \ln \left( 3+2%
\sqrt{2}\right) }$ is also sufficient. Since the function $r\mapsto
A_{r}\left( 1,x\right) $ is increasing, so the function $p\mapsto
F_{p}\left( x\right) $ is decreasing, thus it is suffices to show that $%
F_{p}\left( x\right) >0$ for all $x\in \left( 0,1\right) $ if $p=p_{0}=\frac{%
\ln 2}{\ln \ln \left( 3+2\sqrt{2}\right) }$.

Lemma \ref{Lemma 2.4} reveals that for $p=p_{0}$ there is a unique $\tilde{x}%
_{0}$ to satisfy 
\begin{equation}
f_{p_{0}}\left( x\right) =\ln \frac{x-1+\sqrt{2\left( x^{2}+1\right) }}{x+1}-%
\sqrt{2}\frac{x-1}{\left( x+1\right) \sqrt{x^{2}+1}}\frac{x^{p_{0}}+1}{%
x^{p_{0}-1}+1}=0  \label{3.3}
\end{equation}%
such that the function $x\mapsto F_{p}\left( x\right) $ is strictly
increasing on $\left( 0,\tilde{x}_{0}\right) $ and strictly decreasing on $%
\left( \tilde{x}_{0},1\right) $. It is acquired that for $p_{0}=\frac{\ln 2}{%
\ln \ln \left( 3+2\sqrt{2}\right) }$ 
\begin{eqnarray*}
0 &=&F_{p_{0}}\left( 0^{+}\right) <F_{p_{0}}\left( x\right) \leq
F_{p_{0}}\left( \tilde{x}_{0}\right) \\
0 &=&F_{p_{0}}\left( 1\right) <F_{p_{0}}\left( x_{3}\right) \leq
F_{p_{0}}\left( \tilde{x}_{0}\right) ,
\end{eqnarray*}%
which leads to%
\begin{equation*}
A_{p_{0}}\left( 1,x\right) <N\left( 1,x\right) <\left( \exp F_{p}\left( 
\tilde{x}_{0}\right) \right) A_{p_{0}}\left( 1,x\right) .
\end{equation*}%
Solving the equation (\ref{3.3}) by mathematical computation software we
find that $\tilde{x}_{0}\in \left( 0.15806215485976,0.15806215485977\right) $%
, and then 
\begin{equation*}
\beta _{2}=\exp F_{p}\left( \tilde{x}_{0}\right) \approx \allowbreak
1.\,\allowbreak 013\,8,
\end{equation*}%
which proves the sufficiency and inequalities (\ref{Mb}).
\end{proof}

\section{Corollaries}

From the proof of Lemma \ref{Lemma 2.2}, it is seen that $f_{p}\left(
x\right) >0$ for $x\in \left( 0,1\right) $ if and only if $p\geq 4/3$, which
implies that the inequality 
\begin{equation*}
N\left( 1,x\right) =\frac{x-1}{2\ln \frac{x-1+\sqrt{2\left( x^{2}+1\right) }%
}{x+1}}>\frac{x+1}{2}\sqrt{\frac{x^{2}+1}{2}}\frac{x^{p-1}+1}{x^{p}+1}
\end{equation*}%
holds if and only $p\geq 4/3$. In a similar way, the inequality 
\begin{equation*}
N\left( 1,x\right) <\frac{x+1}{2}\sqrt{\frac{x^{2}+1}{2}}\frac{x^{p-1}+1}{%
x^{p}+1}
\end{equation*}%
is valid if and only if $p\leq 1$. The results can be restated as a
corollary.

\begin{corollary}
The inequalities%
\begin{equation}
\frac{AA_{2}}{\mathcal{L}_{p-1}}<N<\frac{AA_{2}}{\mathcal{L}_{q-1}}
\label{4.0}
\end{equation}%
holds if and only if $p\geq 4/3$ and $q\leq 1$, where $\mathcal{L}_{r}$ is
the Lehmer mean defined by (\ref{Lehmer}).
\end{corollary}

Using the monotonicity of the function defined on $\left( 0,1\right) $ by 
\begin{equation*}
F_{p}\left( x\right) =\ln \frac{N\left( 1,x\right) }{A_{p}\left( 1,x\right) }
\end{equation*}%
given in Lemma \ref{Lemma 2.2}, we can obtain a Fan Ky type inequality but
omit the further details of the proof.

\begin{corollary}
Let $a_{1},a_{2},b_{1},b_{2}>0$ with $a_{1}/b_{1}<a_{2}/b_{2}<1$. Then the
following Fan Ky type inequality%
\begin{equation*}
\frac{N\left( a_{1},b_{1}\right) }{N\left( a_{2},b_{2}\right) }<\frac{%
A_{p}\left( a_{1},b_{1}\right) }{A_{p}\left( a_{2},b_{2}\right) }
\end{equation*}%
holds if $p\geq 4/3$. It is reversed if \thinspace $p\leq 1$.
\end{corollary}

\end{document}